\documentclass[12pt]{amsart}

\usepackage{amsmath}
\usepackage{amssymb, amsfonts,amscd,verbatim,hyperref,cleveref,graphics}
\usepackage{xspace}
\usepackage{tikz}
\usetikzlibrary{matrix,calc,decorations.pathmorphing,shapes,arrows}

\newtheorem{thm}{Theorem}[section]
\newtheorem{prop}[thm]{Proposition}
\newtheorem{cor}[thm]{Corollary}

\numberwithin{equation}{section}

\theoremstyle{definition}
\newtheorem{rem}[thm]{Remark}
\newtheorem{ex}[thm]{Example}
\newtheorem{quest}[thm]{Question}

\def\NN{{\mathbb N}}

\def\RR{{\mathbb R}}
\def\QQ{{\mathbb Q}}

\DeclareSymbolFont{bbold}{U}{bbold}{m}{n}
\DeclareSymbolFontAlphabet{\mathbbold}{bbold}
\def\one{\mathbbold{1}}

\newcommand{\zs}

\newcommand{\Iff}{if\textcompwordmark f\xspace}

\newcommand{\Int}{\operatorname{Int}\nolimits}

\newcommand{\term}[1]{{\textit{\textbf{#1}}}}   

\DeclareMathOperator{\supp}{supp}

\begin{document}

\title[The countable sup property]{The countable sup property for lattices of continuous functions}

\author{M.~Kandi\' c}
\address{Fakulteta za Matematiko in Fiziko,
Univerza v Ljubljani,
Jadranska ulica 19,
SI-1000 Ljubljana,
Slovenija }
\email{marko.kandic@fmf.uni-lj.si}

\author{A.~Vavpeti\v c}
\address{Fakulteta za Matematiko in Fiziko,
Univerza v Ljubljani,
Jadranska ulica 19,
SI-1000 Ljubljana,
Slovenija }
\email{ales.vavpetic@fmf.uni-lj.si}


\thanks{The first author acknowledges the financial support from the
  Slovenian Research Agency (research core funding No. P1-0222). The second author acknowledges the financial support from the Slovenian Research Agency (research core funding No. P1-0292 and J1-7025).}

\begin{abstract}
In this paper we find sufficient and necessary conditions under which vector lattice $C(X)$ and its sublattices $C_b(X)$, $C_0(X)$ and $C_c(X)$ have the countable sup property.
It turns out that the countable sup property is tightly connected to the countable chain condition of the underlying topological space $X$.
We also consider the countable sup property of $C(X\times Y)$. Even when both $C(X)$ and $C(Y)$ have the countable sup property it is possible that $C(X\times Y)$ fails to have it. For this construction one needs to assume the continuum hypothesis.  In general, we present a positive result in this direction and also address the question when $C(\prod_{\lambda\in\Lambda} X_\lambda)$ has the countable sup property. Our results can be understood as vector lattice theoretical versions of results regarding products of spaces satisfying the countable chain condition. We also present new results for general vector lattices that are of an independent interest.
\end{abstract}

\maketitle

\section{Introduction}

In topology, when one deals with continuous functions, there are two possibilities. One can work either by open sets or by nets (generalized sequences). Although dealing with nets is maybe computationally easier, one needs to be cautious since there is a variety of different types of nets and one can often make errors.  When the topology of a given space is metrizable (more general first-countable), the sequential nature of the space enables us to work by sequences instead of nets.
Since order convergence in vector lattices is also defined through nets, one would also like to pass from nets to sequences, of course, if possible. In the setting of vector lattices this notion is called the \term{countable sup property}. It plays an important role in the recent research
in vector and Banach lattices. For example, in \cite{Adeeb:17} it was used to prove that every function in $C(\mathbb R^m)$ is the order limit of an order convergent sequence of piecewise affine functions. Next, in \cite{GTX} authors used it to prove that in some Banach function spaces over  $\sigma$-finite measure spaces convex  Koml\'os sets are norm bounded. Last but not least, in \cite{LiChen} authors proved that the universal completion of a vector lattice with a weak unit and with the countable sup property also has the countable sup property. With this result they proved uo-completeness of the universal completion of some vector lattices (see \cite[Theorem 2.10]{LiChen}).
Recall that a vector lattice $E$ is said to be uo-complete whenever every uo-Cauchy net in $E$ uo-converges in $E$. These days uo-convergence plays a very important role in the research of vector lattices. Although uo-convergence is very exciting on its own, its value shows through its applications in Mathematical finance.
For general results on uo-convergence and its unbounded norm version we refer the reader to \cite{GaoX:14,Gao:14,KMT,GTX,LiChen}. For applications of uo-convergence and its techniques to Mathematical finance we refer the reader to \cite{GLX,GLX2}.

In this paper our interest is the countable sup property itself. Although we present some new general results which also extend some results from \cite{Adeeb:17} and \cite{LiChen}, our main concern is the countable sup property for vector lattices of continuous functions on a given topological space. The paper is structured as follows. In \Cref{prelim} we introduce notation and basic notions needed throughout the text. In \Cref{chain conditions} we
introduce different chain conditions on a topological space $X$ and prove that the existence of a strictly positive functional on $C_b(X)$ implies the weakest of them. In \Cref{cs general properties} we connect the countable sup property of $C(X)$ to chain conditions from \Cref{chain conditions}. It turns out that the countable chain condition of a topological space $X$ implies that $C(X)$ has the countable sup property and that, in general, they are not equivalent. However, they are equivalent when $X\in T_{3\frac 1 2}$. Also, for a metric space $X$, the countable sup property of $C(X)$ is equivalent to separability of $X$. Along the way we extend two results from \cite{Adeeb:17} and \cite{LiChen}. In \Cref{ccc0} we prove that vector lattices $C_c(X)$ and $C_0(X)$ simultaneously have the countable sup property or simultaneously fail to have it. In the last section we consider the vector lattice $C(X\times Y)$. It is possible for both $C(X)$ and $C(Y)$ to have the countable sup property while $C(X \times Y)$ lacks it. This follows under continuum hypothesis from Galvin's example \cite{Galvin} and \Cref{C(X) C_b(X)}. This example also leads to an example of an extremally disconnected compact Hausdorff space $X$ such that $C(X)$ has the countable sup property while $C(X\times X)$ lacks
it (see e.g. \cite{Roy:89}). 
We also prove that whenever $C(X)$ has the countable sup property and $Y$ is separable, then $C(X\times Y)$ has the countable sup property. This result can be considered as a vector lattice version of \cite[Theorem 3.3]{Wiscamb:69}. Last but not least, we also prove that  $C(\prod_{\lambda\in\Lambda} X_\lambda)$ has the countable sup property whenever for each finite family $\Lambda_0\subseteq \Lambda$ the space  $C(\prod_{\lambda\in\Lambda_0} X_\lambda)$ has the countable sup property. Again, this can be considered as a vector lattice version of \cite[Theorem 2.2]{Roy:89}.


\section{Preliminaries}\label{prelim}

Throughout the paper, if not otherwise stated, vector lattices are assumed to be Archimedean. A vector $x$ of a vector lattice $E$ is said to be \term{positive} if $x\geq 0.$ The set of all positive vectors of $E$ is denoted by $E_+$. A vector $e\in E_+$ is said to be a \term{unit} if for every $x\in E$ there is some $\lambda\geq 0$ such that $|x|\leq \lambda e$. A positive vector $e\in E$ is said to be a \term{weak unit} if $|x|\wedge e=0$ implies $x=0$. A vector sublattice $F$ is \term{order dense} in $E$ if for each $x\in E_+$ there is $y \in F_+$ satisfying $0<y\leq x.$ If for each $x\in E$ there is some $y\in F$ with $x\leq y$, then $F$ is said to be \term{majorizing} in $E$. Order dense sublattices are always \term{regular}, i.e., if $x_\alpha\downarrow 0$ in $F$, then $x_\alpha\downarrow 0$ in $E$. It is easy to see that  properties of being an order dense, majorizing or a regular sublattice are transitive relations. Vector lattice $E$ is said to satisfy the \term{countable sup property} whenever every nonempty subset possessing a supremum contains a countable subset possessing the same supremum. The countable sup property is equivalent to the following fact:  for each net $(x_\alpha)$ in $E$ that satisfies $0\leq x_\alpha\uparrow x$ there is an increasing sequence $\alpha_n$ such that $0\leq x_{\alpha_n}\uparrow x$ (see e.g. \cite[Theorem 23.2]{Luxemburg:71}).
If $F$ is a regular sublattice of a vector lattice with the countable sup property, it is easy to see that $F$ has the countable sup property as well.
A positive functional $\varphi$ on $E$ is said to be \term{strictly positive} if $\varphi(x)=0$ for some  vector $x\in E_+$ implies $x=0$.
If there exists a strictly positive functional on $E$, then $E$ has the countable sup property (see e.g. \cite[Theorem 1.45]{Aliprantis:03}).

Let $E$ be a vector lattice and let $E^\delta$ be its order (Dedekind) completion. Then $E$ is an order dense and majorizing sublattice in $E^\delta$, moreover these properties characterize $E^\delta$. If $F$ is a regular sublattice of $E$, then $F^\delta$ is a regular sublattice of $E^\delta$ (see e.g. \cite[Theorem 2.10]{GTX}).
For the unexplained terminology about vector lattices we refer the reader to \cite{Aliprantis:03,Aliprantis:06}.

Since in this paper we will be mainly concerned with the vector lattice $C(X)$, we recall the most needed topological facts. A space $X$ is said to satisfy $T_{3\frac 1 2}$-separation axiom if $X$ is Hausdorff and for each closed set $F$ in $X$ and $x\notin F$ there is a continuous function $f$ on $X$ with $f(x)=1$ and $f\equiv 0$ on $F$. We write $X\in T_{3\frac 1 2}$ for short when $X$ satisfies $T_{3 \frac 1 2}$-separation axiom. For a function $f$ on $X$, the set $\mathcal Z(f)$ of zeros of $f$ is called the \term{zero set} of $f$. The closure of  $X\setminus \mathcal Z(f)$ is called the \term{support} of $f$ and is denoted by $\supp f.$ A subset $U$ of $X$ is said to be a \term{cozero set} whenever there exists a function $f\in C(X)$ such that $\mathcal Z(f)=X\setminus U$. The constant $1$ function is denoted by $\one$.
The order on $C(X)$ is defined pointwise. Clearly $\one$ is a unit in $C_b(X)$ while it is only a weak unit in $C(X)$. Also, when $X$ is not compact, then $\one \notin C_0(X)$.

%



\section{Chain conditions in topological spaces}\label{chain conditions}

Countability in Topology has a special role. One of the first properties one encounters is \term{separability} of a topological space $X$, i.e., the existence of a countable dense set in $X$. Although the notion of separability is very natural and intuitive, in Functional analysis there is an abundance of nonseparable Banach spaces which are of great importance. As an illustration, consider a compact Hausdorff space $K$. It is well-known that $C(K)$ is separable \Iff $K$ is metrizable (see e.g. \cite[Proposition 2.1.8]{Meyer-Nieberg:91}). For real or abstract analysis sometimes the full force of separability is not needed. In those cases, one can develop richer theory under relaxed assumptions however through more involved and deeper proofs. One can relax the notion of separability through the so-called ``chain conditions" imposed on a topological space $X$ which are introduced below.

A topological space $X$ is said to satisfy
\begin{itemize}
  \item \term{Shanin's condition} if every point-countable family of open sets is countable, i.e. every point is a member of only countable many members of a family,
  \item \term{calibre $\aleph_1$} if every uncountable family of nonempty open sets contains an uncountable subfamily with a nonempty intersection,
  \item \term{Knaster's condition} or \term{property K} if every uncountable set of nonempty open sets contains an uncountable subset in which no two elements are disjoint,
  \item \term{countable chain condition} if  every family of pairwise disjoint open subsets of $X$ is countable.
\end{itemize}

We will write for short that $X$ is CCC whenever $X$ satisfies the countable chain condition.
In general, we have the following chain of implications regarding chain conditions.
$$\textrm{separability} \Rightarrow \textrm{Shanin's condition}\Rightarrow \textrm{calibre }\aleph_1 \Rightarrow \textrm{property K} \Rightarrow\textrm{CCC}$$
 and none of them can be reversed. When $X$ is metrizable, CCC implies separability (see e.g. \cite[Theorem 1]{VR:59}). Hence, in the metrizable case, all conditions that were introduced above are equivalent.

Countable chain condition is also called \term{Suslin's condition} and it originates from the so-called \term{Suslin's problem} about totally ordered sets posed by Suslin \cite{Suslin}.
Suslin asked whether there exists a totally ordered set $R$ which is not order isomorphic to the real line $\mathbb R$ and satisfies the following properties:
\begin{enumerate}
  \item $R$ does not have a least nor a greatest element;
  \item between any two elements there is another;
  \item every nonempty bounded set has a supremum and an infimum;
  \item every collection of pairwise disjoint nonempty open intervals in $R$ is countable.
\end{enumerate}
Totally ordered set which satisfies (1)-(4) and is not order isomorphic to $\mathbb R$ is called a \term{Suslin line}. \term{Suslin's hypothesis} says that there are no Suslin lines.  Solovay and Tennenbaum \cite{Solovay} proved that Suslin's hypothesis is undecidable in ZFC. However, Suslin's hypothesis is true \cite{Solovay} if one assumes the negation of the continuum hypothesis and \term{Martin's axiom} \cite{Martin} which says that no compact Hausdorff space is a union of fewer than $2^{\aleph_1}$ nowhere dense sets. Regarding Suslin's hypothesis we refer the reader also to \cite{Jech} and \cite{Tennenbaum}.

So far we have seen that separability condition and  CCC are the strongest and the weakest chain condition we introduced so far, respectively.
The CCC condition can be relaxed as follows. The space $X$ is said to satisfy \term{cozero CCC} or \term{CCC for cozero sets} if any collection of pairwise disjoint cozero sets is countable. Obviously CCC implies CCC for
cozero sets. We will see later that a space $X$ which satisfies CCC for cozero sets does not necessary satisfy CCC (see \Cref{strictly positive -> co zero CCC} and \Cref{CCC cozero strictly positive functional}).

%

 In general, even compact Hausdorff spaces do not satisfy CCC.  The example of such space can be found in \cite[p. 116-117]{Steen:70}.
 On the other hand compact groups satisfy Shanin's condition which is stronger than CCC. This  follows from the fact that every Cantor cube satisfies Shanin's condition and from the remarkable Ivanovski\u\i--Vilenkin--Kuzminov theorem which says that compact groups are continuous images of Cantor cubes. More generally, Tka\v cenko proved that every $\sigma$-compact group satisfies CCC. For details see \cite{Tod:97}. Compact groups share another feature which is of more use to us within the framework of this paper.  If $G$ is a compact Hausdorff group, then the normalized Haar measure $\lambda$ on $G$ induces a strictly positive functional $f\mapsto \int_G fd\lambda$ on $C(G)$. In general, a compact Hausdorff space $K$ does not admit strictly positive finite Borel measures.
 When they do, then such measures induce strictly positive functionals on $C(K)$.

\begin{prop}\label{strictly positive -> co zero CCC}
If there exists a strictly positive functional on $C_b(X)$, then $X$ satisfies CCC for cozero sets.
\end{prop}

\begin{proof}
 Let $\varphi$ be a strictly positive functional on $C_b(X)$ and let $(U_\lambda)$ be a family of pairwise disjoint cozero sets. Then for each $\lambda$ there exists a function $f_\lambda\colon X\to [0,1]$ such that $f_\lambda$ is nonzero precisely on $U_\lambda$.
Hence, if $f$ is a finite sum of some functions in $(f_\lambda)$, then $0\leq f\leq \one.$  For each $n\in\mathbb N$ define $A_n=\{f_\lambda:\; \varphi(f_\lambda)\geq \frac{1}{n}\}.$ Suppose $A_m$ is infinite for some $m\in\mathbb N$. Pick $k\cdot m$ distinct functions in $A_m$ and denote their sum by $f_0$. Then $0\leq f_0\leq \one$ implies
$\varphi(\one)\geq \varphi(f_0)\geq k.$ This contradicts the fact that $\varphi$ is a functional. Hence, each set $A_n$ is finite from where it follows that the family $(f_\lambda)$ is countable.
\end{proof}

By applying \Cref{C(X) C_b(X)} we conclude that although CCC or CCC for cozero sets are weaker notions than separability, yet they are still strong enough to provide the right tool for order analysis on the vector lattice $C_b(X)$.

The following example shows that there exists a Hausdorff space $X$ such that $C_b(X)$ admits a strictly positive functional; yet $X$ does not satisfy CCC. This example shows that CCC for cozero sets is weaker than CCC.

\begin{ex}\label{CCC cozero strictly positive functional}
Let us define the topology $\tau$ on $\RR^2$ as follows:
a point $(x,y)\in(\RR\times\QQ)\cup(\QQ\times\RR)$ has Euclidean neighborhoods, but
a point $(x,y)\in(\RR\setminus\QQ)^2$ 
has
$$\{((x-r,x+r)\cap(\RR\setminus\QQ))\times\{y\}:\; r>0\}$$ for its fundamental system. The resulting space is denoted by $X$. Because $\tau$ is stronger
than Euclidean topology, continuous functions on $X$ separate the points. In particular, $X$  is Hausdorff.
Because $\{\RR\setminus\QQ\times\{y\}:\; y\in\RR\setminus\QQ\}$ is an uncountable collection of disjoint
open sets, the space $X$ is not CCC.

We claim that if $f\colon X\to \mathbb R$ is a nonzero $\tau$-continuous function, then $f$ is nonzero on some Euclidean ball. First, we may assume $f\geq 0$. Suppose $f$ is nonzero at a point $(\widetilde x,\widetilde y)\in X$ and denote $c:=f(\widetilde x,\widetilde y)>0$. If $(\widetilde x,\widetilde y)\in (\RR\times\QQ)\cup(\QQ\times\RR)$, then by continuity $f$ is nonzero on some Euclidean neighborhood of $(\widetilde x,\widetilde y)$. Suppose now $\widetilde x$ and $\widetilde y$ are both irrational. Again, by continuity of $f$ we can find $\delta>0$ such that for each $x\in (\widetilde x-\delta,\widetilde x+\delta)\cap (\mathbb R\setminus \mathbb Q)$ we have $f(x,\widetilde y)\geq \tfrac c 2.$ Pick $x_\star\in (\widetilde x-\delta,\widetilde x+\delta)\cap \mathbb Q$. Since $(x_\star,\widetilde y)$ has Euclidean neighborhoods, continuity of $f$ implies that there is $\delta_\star>0$ such that
for each $(x,y)\in K((x_\star,\widetilde y),\delta_\star)$ we have $f(x_\star,\widetilde y)\geq \tfrac c 4.$ This proves the claim.

We claim that there exists a strictly positive functional on $C_b(X)$. Let $(q_n)$ be one of the enumerations of the countable set $\mathbb Q^2$. Then the well-defined mapping $\varphi\colon C_b(X)\to \mathbb R$ given  by $\varphi(f)=\sum_{n=1}^\infty \tfrac{1}{2^n}f(q_n)$  defines a positive functional on $C_b(X)$.
If $f\in C_b(X)$ is nonzero, then $f(q_n)>0$ for some $q_n\in \mathbb Q^2$. Hence, $\varphi(f)\geq \tfrac{1}{2^n}f(q_n)>0$, so that $\varphi$ is strictly positive.
%
%
\end{ex}

Although CCC for cozero sets is weaker than CCC, by \Cref{CCC = countabe sup} they coincide when $X\in T_{3\frac 12}$.
Similar argument as in \Cref{CCC cozero strictly positive functional} shows that $C_b(X)$ admits strictly positive functionals when $X$ is separable.

%
%
%

\section{The countable sup property}\label{cs general properties}

As we already observed, if $F$ is a regular sublattice of $E$ and $E$ has the countable sup property, then $F$ itself has the countable sup property. If $F$ is not regular in $E$, then directly from the definition it is not clear whether $F$ still has the countable sup property since  $0\leq x_\alpha\uparrow x$ in $F$ does not necessary imply $0\leq x_\alpha\uparrow x$ in $E$. However, there is a way around through the following theorem whose proof can be found in \cite[Theorem 29.3]{Luxemburg:71}.
Recall that in \Cref{prelim} we assumed that all vector lattices are Archimedean.

\begin{thm}\label{countable sup disjoint system}
A vector lattice $E$ has the countable sup property \Iff every disjoint system of positive vectors which is bounded from above is countable.
\end{thm}

Now it easy to see that sublattices of vector lattices with the countable sup property have the countable sup property. The following question appears naturally.

\begin{quest}\label{question csup sublattice}
  Suppose $F$ is a vector sublattice of a vector lattice $E$ and suppose $F$ has the countable sup property. Does $E$ have the countable sup property?
\end{quest}

The following example reveals that the answer can be negative  even if $F$ is an order dense ideal in $E$.

\begin{ex}\label{example E}
Let $\Omega$ be an uncountable set and $\ell_\infty(\Omega)$ be the vector lattice of all bounded functions on $\Omega$.
Let $E$ be the vector lattice of all functions in $\ell_\infty(\Omega)$ with a countable support. Then $E$ is an order dense ideal in $\ell_\infty(\Omega)$, $E$ has the countable sup property while $\ell_\infty(\Omega)$ does not have it.
\end{ex}

Note that in \Cref{example E} the vector lattice $E$ has the countable sup property and contains  an uncountable disjoint system of positive vectors $\{\chi_\omega:\; \omega\in \Omega\}$. This is impossible when a given vector lattice possesses weak units.

\begin{prop}\label{countable sup weak unit}
 A vector lattice $E$ with a weak unit has the countable sup property \Iff every disjoint system of positive vectors is countable.
\end{prop}

\begin{proof}
Suppose $E$ has the countable sup property. Let $e$ be a weak unit in $E$. Choose a disjoint system $(x_\lambda)_{\lambda\in\Lambda}$ of positive vectors in $E$. Then $(x_\lambda\wedge e)_{\lambda\in\Lambda}$ is an order bounded disjoint system of positive vectors in $E$ as well. Since $E$ has the countable sup property,  the system $(x_\lambda\wedge e)_{\lambda\in\Lambda}$ is countable. Hence, $x_\lambda\wedge e=0$ for all but countably many $\lambda\in \Lambda.$ Since $e$ is a weak unit, we have $x_\lambda=0$ for all but countably many $\lambda\in \Lambda.$

The converse implication is clear by \Cref{countable sup disjoint system}.
\end{proof}

Now we are able to provide some positive answers to \Cref{question csup sublattice}.
 Suppose $E$ is a vector lattice which contains a sublattice $F$ with the countable sup property. Then $E$ has the countable sup property in the following cases:
\begin{itemize}
  \item $E=F^\delta$ is the order completion of $F$;
  \item $F$ has a weak unit and $E=F^u$ is the universal completion of $F$.
\end{itemize}

For the proofs of these facts we refer the reader to \cite[Theorem 32.9]{Luxemburg:71} and \cite[Lemma 2.9]{LiChen}, respectively.
In the following theorem we extend the results above to a more general setting.

\begin{thm}\label{weak unit countable sup whole lattice}
Let $F$ be an order dense vector sublattice of a vector lattice $E$. If $F$ has the countable sup property, then $E$ has the countable sup property in the following cases:
\begin{enumerate}
\item [(a)] $F$ is majorizing in $E$.
\item [(b)] $F$ has a weak unit.
\end{enumerate}
\end{thm}

\begin{proof}
(a) Since $F$ is regular in $E$, we conclude that $F^\delta$ is  regular in $E^\delta.$ Also, order density of $F$ in $E$ implies that $F^\delta$ is order dense in $E^\delta$. Since $F^\delta$ is order complete, by \cite[Theorem 1.40]{Aliprantis:03} $F^\delta$ is an ideal in $E^\delta$. Since $F^\delta$ is also majorizing in $E^\delta$, we conclude $E^\delta=F^\delta.$ Now we apply the fact that the countable sup property is preserved under passing to order completions. Hence, $E$ has the countable sup property. 

(b) Let $u$ be a weak unit in $F$. Then order density of $F$ in $E$ implies that $u$ is also a weak unit in both $E$ and $E^\delta$.

Suppose first that $F$ is an ideal in $E$. Let $(x_\lambda)_{\lambda\in\Lambda}$ be a disjoint system of positive vectors in $E$. Then $(x_\lambda\wedge u)_{\lambda\in\Lambda}$ is a disjoint system of positive vectors in $F$. Since $F$ has the countable sup property, $x_\lambda\wedge u=0$ for all but countably many $\lambda\in\Lambda$. Since $(x_\lambda\wedge nu)\uparrow_n x_\lambda$ in $E$, we have $x_\lambda=0$ for all but countably many $\lambda\in \Lambda$. \Cref{countable sup weak unit} implies that $E$ has the countable sup property.

For the general case, consider order completions $F^\delta$ and $E^\delta$ of $F$ and $E$, respectively. As in (a) we argue that $F^\delta$ is an order dense ideal in $E^\delta.$ Since the countable sup property is preserved under order completions, the first part of the proof implies that $E^\delta$ has the countable sup property. Now, $E$ as a sublattice of $E^\delta$ also has the countable sup property. Finally we conclude that $E$ has the countable sup property. 
\end{proof}

It should be noted here that the vector lattice $E$ in \Cref{example E} does not have weak units and is not majorizing in $\ell_\infty(\Omega)$.

Suppose $E$ is a vector lattice with the property that each disjoint system of positive vectors in $E$ is countable. In view of \Cref{countable sup weak unit} it is reasonable to suspect that $E$ has a weak unit. However, the normed vector lattice
$c_{00}$ of all eventually zero sequences does not have weak units, yet each disjoint system of positive vectors in $c_{00}$ is countable. In Banach lattices this property implies the existence of weak units. We even have a stronger result.

\begin{thm}
Let $E$ be a completely metrizable locally solid vector lattice.
Every disjoint system of positive vectors of $E$ is countable \Iff $E$ has a weak unit and the countable sup property.
\end{thm}

\begin{proof}
Assume that every disjoint system of positive vectors of $E$ is countable. By \Cref{countable sup disjoint system} it follows that $E$ has the countable sup property. To prove that $E$ has a weak unit, pick a maximal disjoint system $\mathcal F$ of nonzero positive vectors. Such systems exist by Zorn's lemma. By the assumption $\mathcal F$ is countable.

Suppose first that $\mathcal F=(f_n)_{n\in\mathbb N}$.
Since the topology of $E$ is metrizable, there exists a local basis $(U_n)$ of zero consisting of solid sets such that $U_{n+1}\subseteq U_n$ for each $n\in\mathbb N.$

We claim that there exists a local basis $(V_n)$ of zero consisting of solid sets such that $V_{n+1}+V_{n+1}\subseteq V_n$ for each $n\in\mathbb N.$ We first define $V_1=U_1$. Suppose that $V_n$ is already defined for some $n\in\mathbb N$. Since addition is continuous in $E$, there exist solid neighborhoods $S$ and $S'$ of zero with $S+S'\subseteq V_n.$ Then $V_{n+1}:=S\cap S'\cap U_{n+1}$ satisfies $V_{n+1}+V_{n+1}\subseteq V_n.$ Since for each $n\in\mathbb N$ we have $V_n\subseteq U_n$, the family $(V_n)$ is a local basis for zero.

For each $n\in \mathbb N$ find $\lambda_n>0$ such that $f_n\in \lambda_n V_n$ and define $e_n:=\frac{f_n}{\lambda_n}\in V_n$. For each $n\in\mathbb N$ we define $s_n=e_1+\cdots+e_n.$
We claim that $(s_n)$ is Cauchy in $E$. To see this, pick a neighborhood $V$ of zero. Then there is $n_0\in \mathbb N$ such that $V_{n_0}\subseteq V.$ If $n>m\geq n_0,$ then
$$s_n-s_m=e_{m+1}+\cdots+e_n\in V_{n-(n-m)}=V_m\subseteq V_{n_0}\subseteq V.$$ Since $E$ is complete the sequence $(s_n)$ converges to some element $e$. Since the cone of $E$ is closed by \cite[Theorem 2.21]{Aliprantis:03}, we have $0\leq s_n\leq e$ for each $n\in \mathbb N$.

Suppose $x\wedge e=0$. Then for each $n\in \mathbb N$ we have
$$0\leq x\wedge e_n\leq x\wedge s_n\leq x\wedge e=0.$$
Maximality of $(e_n)$ implies $x=0$; hence $e$ is a weak unit in $E$.

If $\mathcal F$ is finite, then we define the vector $e=\sum_{f\in \mathcal F}f.$ Again, maximality of $\mathcal F$ implies that $e$ is a weak unit in $E$.

The converse statement follows from \Cref{countable sup weak unit}.
\end{proof}

Although the countable sup property of a vector lattice $C_b(X)$ is a purely  order theoretical notion, it can be connected to a particular chain condition of $X$.

\begin{prop}\label{C(X) C_b(X)}
For a topological space $X$ the following statements are equivalent: \begin{enumerate}
\item[(a)] $C(X)$ has the countable sup property.
\item[(b)] $C_b(X)$ has the countable sup property.
\item[(c)] $X$ satisfies CCC for cozero sets.
\end{enumerate}
\end{prop}

\begin{proof}
Since $C_b(X)$ is an order dense ideal in $C(X)$, the equivalence between (a) and (b) follows from \Cref{weak unit countable sup whole lattice}.

(a)$\Rightarrow$(c) Let $(U_\lambda)$ be a family of disjoint cozero sets in $X$. For each $\lambda$ there exists a nonnegative function $f_\lambda\colon X\to [0,1]$ such that $f_\lambda$ is nonzero precisely on $U_\lambda.$ Hence, $(f_\lambda)$ is a family of pairwise disjoint functions in $C(X)$ which is bounded by above by $\one$. Countable sup property of $C(X)$ implies that $(f_\lambda)$ is countable.

(c)$\Rightarrow$(a) Suppose $(f_\lambda)$ is a family of nonnegative pairwise disjoint functions. For each $\lambda$ define $U_\lambda=f_\lambda^{-1}((0,\infty))$. Then $(U_\lambda)$ is a family of pairwise disjoint cozero sets in $X$. By the assumption this family is countable.
\end{proof}

If $X$ is CCC, then $X$ satisfies CCC for cozero sets, so that  by \Cref{C(X) C_b(X)} the lattice $C(X)$ has the countable sup property.
The space $X$ from \Cref{CCC cozero strictly positive functional} is an example of a space which is not CCC, yet $C_b(X)$ has the countable sup property.
If $X\in T_{3\frac 1 2}$, then $X$ is CCC \Iff $X$ satisfies CCC for cozero sets.
%

\begin{thm}\label{CCC = countabe sup}
For a topological space $X\in T_{3\frac 1 2}$ the following assertions are equivalent.
\begin{enumerate}
\item[(a)] $X$ is CCC.
\item[(b)] $X$ satisfies CCC for cozero sets.
\item[(c)] $C(X)$ has the countable sup property.
\end{enumerate}
If $C_b(X)$ admits a strictly positive functional, then (a), (b) and (c) hold.
\end{thm}

\begin{proof}
That (a) implies (b) is obvious and that (b) implies (c) follows from \Cref{C(X) C_b(X)}. To see that (c) implies (a), pick  a family $(U_\lambda)$ of nonempty disjoint open sets in $X$. For each $\lambda$ we find a nonzero  function $f_\lambda\colon X\to [0,1]$ such that $f_\lambda\equiv 0$ on $X\setminus U_\lambda.$ Then $(f_\lambda)$ is a family of pairwise disjoint positive functions which are bounded by $\one$. The countable sup property of $C(X)$ implies that $(f_\lambda)$ is countable.
%
\end{proof}

Since metrizable spaces are $T_{3\frac 1 2}$, the following result immediately follows from \Cref{CCC = countabe sup}.

\begin{cor}\label{separable=CCC}
A metrizable space $X$ is separable \Iff $C(X)$ has the countable sup property.
\end{cor}

When the underlying topological space in \Cref{separable=CCC} is a topological group, the countable sup property is connected to $\aleph_0$-boundedness of the group.
Recall that a topological group $G$ is said to be \term{$\aleph_0$-bounded} whenever for every open set $U$ there is a countable set $S$ such that $$G=U\cdot S=\bigcup_{s\in S}U\cdot s.$$

\begin{thm}\label{groups CCC}
Let $G$ be a metrizable topological group. Then the following assertions are equivalent.
\begin{enumerate}
\item[(a)] $G$ is separable.
\item[(b)] $G$ is CCC.
\item[(c)] $G$ is $\aleph_0$-bounded.
\item[(d)] $C(G)$ has the countable sup property.
\end{enumerate}
\end{thm}

\begin{proof}
Since $G$ is metrizable, (a) and (b) are equivalent. Also, metrizability of $G$ implies $G\in T_{3\frac 1 2}$, so that (b) and (d) are equivalent by \Cref{CCC = countabe sup}. That (b) implies (c) follows from  \cite[Proposition 3.3]{Tkachenko}.

(c)$\Rightarrow$(b) Since $G$ is metrizable, the unit element has a countable local basis $(U_n)$ consisting of symmetric sets. For each $n\in\mathbb N$ there is a countable set $S_n$ such that $G=U_n\cdot S_n.$ We claim that the countable set $S:=\bigcup_{n=1}^\infty S_n$ is dense in $G$. To see this, pick an open set $U$ in $G$ and $x\in U$. Then there is $n\in \mathbb N$ such that $U_nx\subseteq U.$ Also, there are $u\in U_n$ and $s\in S_n$ such that $x=us.$ Since
$s=u^{-1}us=u^{-1}x\in U_nx\subseteq U$, we conclude that $S$ is dense in $G$.
\end{proof}

As a special case of \Cref{groups CCC} we consider normed spaces. It is well-known that a normed space $X$ is separable whenever its norm dual $X^*$ is separable. Applying \Cref{groups CCC} one sees that $X$ is CCC whenever $X^*$ is CCC.
Therefore, maybe the easiest example of a Banach space which is not CCC is $\ell_\infty.$ This example also shows that CCC property is not closed under taking norm duals.

\Cref{separable=CCC} can be used as an argument why $C(\mathbb R^n)$ has the countable sup property. We actually do not require the full force of separability of $\mathbb R^n$. In \cite[Theorem 5.4]{Adeeb:17} authors argued as follows: the space $\mathbb R^n$ can be exhausted by  sets $\{m B_\infty^n:\; m\in\mathbb N\}$ where $B_\infty^n$ denotes the closed unit ball in $\mathbb R^n$ with respect to $\|\cdot\|_\infty$-norm. Since the Riemann integral is a strictly positive functional on $C(m B_\infty^n)$ for each $m\in \mathbb N$ the space $C(m B_\infty^n)$ has the countable sup property. By exhausting $\mathbb R^n$ with closed balls $m B_\infty^n$ they proved that $C(\mathbb R^n)$ itself has the countable sup property. It should be noted here that $C(\mathbb R^n)$ does not admit a strictly positive functional.

We proceed to the extension of \cite[Theorem 5.4]{Adeeb:17}.

\begin{prop}\label{exhaust}
If there exist sets $(A_n)$ in $X$ such that $C(A_n)$ has the countable sup property and $\bigcup_{n=1}^\infty A_n$ is dense in $X$, then $C(X)$ has the countable sup property.
\end{prop}

\begin{proof}
Suppose $(f_\lambda)_{\lambda\in\Lambda}$ is a positive disjoint system in $C(X)$. Then
$(f_\lambda|_{A_n})_{\lambda\in\Lambda}$ is a positive disjoint system in $C(A_n)$ which is countable by the assumption. Hence, the set
$$\{\lambda\in\Lambda:\; f_\lambda|_{A_n}\neq 0 \textrm{ for some }n\in\mathbb N\}$$ is countable. This implies that $f_\lambda=0$ on $\bigcup_{n=1}^\infty A_n$ for all but countably many $\lambda\in \Lambda.$ Since $\bigcup_{n=1}^\infty A_n$ is dense in $X$, we conclude $f_\lambda=0$ for all but countably many $\lambda\in \Lambda.$
\end{proof}

\begin{cor}
Let $A$ be a dense subspace of a topological space $X$.
If $C(A)$ has the countable sup property, then $C(X)$ has the countable sup property.
\end{cor}

 Suppose $U$ is a dense set in $X$. It is easy to see that $X$ is CCC \Iff $U$ is CCC. Hence, when $X\in T_{3 \frac 1 2}$ the lattice $C(X)$ has the countable sup property \Iff $C(U)$ has the countable sup property. The following example shows that this does not hold when $X\notin T_{3\frac 1 2}.$

\begin{ex}
Let $X$ be the topological space from \Cref{CCC cozero strictly positive functional}. Since $C_b(X)$ has the countable sup property and $X$ is not CCC, by \Cref{CCC = countabe sup} we conclude $X\notin T_{3\frac 1 2}.$  It is easy to see that the open set $U=(\mathbb R\setminus \mathbb Q)^2$ is dense in $X$.
Since for each $\lambda\in \mathbb R\setminus \mathbb Q$ the set $\mathbb R\setminus \mathbb Q \times \{\lambda\}$ is both open and closed in $X$, the characteristic function $f_\lambda$ of $\mathbb R\setminus \mathbb Q\times \{\lambda\}$ is continuous. Hence,
$\{f_\lambda:\; \lambda\in \mathbb R\setminus \mathbb Q\}$ is an uncountable positive disjoint system in $C(U)$. This shows  that $C(U)$ does not have the countable sup property.
\end{ex}

The following example shows that even when $X\in T_{3\frac 1 2}$ and $C(X)$ has the countable sup property, one can find a subset $A\subseteq X$ such that $C(A)$ does not have the countable sup property.

\begin{ex}
Let $\RR_S$ be the Sorgenfrey line, i.e. the set $\RR$ with the topology defined by a basis $\{[a,b):\; a,b\in\RR\}$.
Since $\RR_S\in T_{3\frac 1 2}$, we conclude that the product $\RR_S\times \RR_S\in T_{3\frac 1 2}$. Since $\QQ^2$ is dense in $\RR_S\times \RR_S$,  the space $\RR_S\times \RR_S$ is separable and therefore it is CCC. Hence, $C(\RR_S\times \RR_S)$ has the countable sup property.
On the other hand, the set $A:=\{(x,-x):\; x\in\RR\}$ is uncountable and discrete. This implies $A$ is not CCC.  Since the property $T_{3\frac 1 2}$ is hereditary, $C(A)$ does not have the countable sup property by \Cref{CCC = countabe sup}.
\end{ex}


Although in \Cref{exhaust} we exhausted the space $X$, we did not exhaust the vector lattice $C(X)$ since $C(A)$ is not even a subset of $C(X)$. However, if one considers the case when $X\in T_{3\frac 1 2}$ the restriction operator $\Phi_A\colon f\mapsto f|_A$ induces a lattice isomorphism between $C(X)/\ker \Phi_A$ and an order dense sublattice of $C(A)$ (see e.g. \cite[Proposition 3.5]{KV}). This remark and \Cref{exhaust} lead us to the following general ``exhaustion-type" result for the countable sup property of general vector lattices.

\begin{prop}\label{quotients CS}
Let $E$ be a vector lattice and suppose there exist countably many uniformly closed ideals $J_n$ in $E$ such that $E/J_n$ has the countable sup property for each $n\in\mathbb N$. If $\bigcap_{n=1}^\infty J_n=\{0\}$, then $E$ has the countable sup property.
\end{prop}

\begin{proof}
Pick an order bounded disjoint system $(x_\lambda)_{\lambda\in\Lambda}$ of positive vectors in $E$. Then $(x_\lambda+J_n)_{\lambda\in\Lambda}$ is an order bounded disjoint system of positive vectors in $E/J_n.$ Since $J_n$ is uniformly complete, $E/J_n$ is Archimedean, so that \Cref{countable sup disjoint system} implies that $(x_\lambda+J_n)_{\lambda\in\Lambda}$ is countable.
Let us define $$\Lambda'=\{\lambda\in \Lambda:\; x_\lambda+J_n\neq 0 \textrm{ for some }n\in\mathbb N\}.$$ Then
$(x_\lambda)_{\lambda\in\Lambda'}$ is a countable disjoint system of positive vectors. If $x_\mu\wedge x_\lambda=0$ for each $\lambda\in\Lambda'$, then $x_\mu\in J_n$ for each $n\in\mathbb N$, so that $x_\mu\in \bigcap_{n=1}^\infty J_n=\{0\}.$ This implies that $E$ has the countable sup property.
\end{proof}

\section{Vector lattices $C_c(X)$ and $C_0(X)$} \label{ccc0}

As was already observed in \Cref{C(X) C_b(X)}, a vector lattice $C(X)$ has the countable sup property \Iff $C_b(X)$ has the countable sup property. This followed from the fact that $C_b(X)$ is an order dense ideal in $C(X)$ and that $\one$ is a (weak) unit in $C_b(X)$. In this section we are interested in finding necessary and sufficient conditions for the lattices $C_c(X)$ and $C_0(X)$ to have the  countable sup property in terms of properties of $X$. In order to establish the connection between CCC property of $X$ and the countable sup property of $C_0(X)$, we first characterize  weak units in $C_0(X)$.

\begin{ex}
Let $\Omega$ be a set equipped with the discrete topology. Then the vector lattice $C_0(\Omega)$  has weak units \Iff $\Omega$ is countable.
\end{ex}
\begin{prop}\label{C_0 CS}
Let $X$ be a locally compact Hausdorff space.
 \begin{enumerate}
 \item[(a)] A nonnegative function $f\in C_0(X)$ is a weak unit in $C_0(X)$ \Iff $\Int\mathcal Z(f)=\emptyset.$
 \item[(b)] $C_0(X)$ has a weak unit \Iff $X$ contains an open dense $\sigma$-compact subspace.
\end{enumerate}
\end{prop}

\begin{proof}
(a) Since $f$ is continuous, $\mathcal Z(f)$ is a closed subset of $X$. Suppose $\Int \mathcal Z(f)=\emptyset$. If $g\wedge f=0$, then $g\equiv 0$ on $X\setminus \mathcal Z(f)$. Since $X\setminus \mathcal Z(f)$ is dense in $X$ and $g$ is continuous, we have $g=0$.

Assume now that $\Int\mathcal Z(f)$ is nonempty and pick $x\in \Int \mathcal Z(f)$. By Urysohn's lemma for $C_0(X)$-spaces, there exists a nonnegative function $g\in C_0(X)$ with $g(x)=1$ and $g\equiv 0$ on $X\setminus \Int \mathcal Z(f)$. Then $f\wedge g=0$ and since $g\neq 0$ we conclude that $f$ is not a weak unit in $C_0(X)$.

(b) If $f\in C_0(X)$ is a weak unit, then $\Int \mathcal Z(f)=\emptyset.$ Since for each $\epsilon>0$ the set
$\{x\in X:\; f(x)\geq \epsilon\}$ is compact in $X$, the dense set
$$X\setminus \mathcal Z(f)=\bigcup_{n=1}^\infty \{x\in X:\; f(x)\geq \tfrac 1 n\}$$
is $\sigma$-compact.

For the converse, assume $U\subseteq X$ is a $\sigma$-compact open dense subspace in $X$ and define $F=X\setminus U$.

Since $U$ is open in $X$, it is locally compact and since it is $\sigma$-compact, \cite[Theorem X.7.2]{Dugundji:66} implies that there exists an increasing sequence of relatively compact open sets $(U_n)$ in $U$ such that $\overline{U_n}\subseteq U_{n+1}.$ By Urysohn's lemma for $C_0(X)$-spaces there exists a function $f_n:X\to[0,1]$ with compact support such that $f_n\equiv 1$ on $\overline{U_n}$ and $f\equiv 0$ on $X\setminus U_{n+1}.$ The function $f:=\sum_{n=1}^\infty \frac{f_n}{2^n}$ is strictly positive on $U$ and zero on $F$. By (a) we conclude that $f$ is a weak unit in $C_0(X)$.
\end{proof}

\begin{rem}
If $X$ is not compact, then $C_c(X)$ does not have a weak unit. Indeed, if $f$ would be a weak unit in $C_c(X)$, then the set $X\setminus \mathcal Z(f)$ is an open $\sigma$-compact dense set in $X$. Since $\supp f$ is compact, we obtain that
$X=\overline{X\setminus \mathcal Z(f)}=\supp f$ is a compact set.
\end{rem}



In the following example we construct a locally compact Hausdorff space $X$ which is not $\sigma$-compact, yet $C_0(X)$ has a weak unit. This example shows that, in general, the existence of an open dense $\sigma$-compact set of a locally compact space does not imply that the whole space is $\sigma$-compact.

\begin{ex}
Let $X=\RR$ and for every $x\not\in\QQ$ we pick a sequence of rational numbers $(q_n(x))_{n\in\NN}$ with
$\lim_{n\to\infty} q_n(x)=x$. Let $\tau$ be the topology on $X$ such that $\{x\}$ is open for every $x\in\QQ$
and
$$\{U_n(x)=\{x,q_n(x),q_{n+1}(x),\ldots\}\mid n\in\NN\}$$ is a fundamental system for $x\not\in\QQ$.
The space $X$ is Hausdorff and locally compact.

Let $A\subset X$ be such that $A\cap (\RR\setminus \QQ)=\{x_\lambda:\; \lambda\in\Lambda\}$ is infinite.
Then $\{U_1(x_\lambda):\; \lambda\in\Lambda\}\cup\{\{x\}:\; x\in A\cap\QQ\}$ is an open cover of $A$.
Because $U_1(x_\lambda)$ is the only element of the cover containing $x_\lambda$, there is no
finite subcover of $A$. Hence every compact subset of $X$ contains only finitely many irrational numbers.
Therefore $X$ is not $\sigma$-compact while $\QQ\subset X$ is $\sigma$-compact, open and dense in $X$.

Let $(r_n)$ be an enumeration of the rational numbers. Since the function $f=\sum_{n=1}^\infty \tfrac{1}{2^n}\chi_{\{r_n\}}$ is continuous
on $X$, strictly positive on $\mathbb Q$ and zero on $\mathbb R\setminus \mathbb Q$, by \Cref{C_0 CS} $f$ is a weak unit in $C_0(X)$.
\end{ex}

In the rest of this section we consider when $C_0(X)$ has the countable sup property. If $X$ is CCC, then $C(X)$ has the countable sup property. Therefore, $C_0(X)$ as a sublattice of $C(X)$,  itself has the countable sup property. The following is an example of a space $X\in T_{3\frac 12}$ which is not CCC while $C_0(X)$ has the countable sup property.

\begin{ex}
Let $\Omega$ be an arbitrary uncountable set equipped by the discrete topology.
Then $\Omega$ is a locally compact Hausdorff space which is not CCC.
Since each $f\in C_0(\Omega)$ has a countable support, $C_0(\Omega)$ has the countable sup property.

\end{ex}

It should be noted that $C_0(\Omega)$ from the preceding example does not have a weak unit.

\begin{prop}
Let $X$ be a locally compact Hausdorff space.
If there is an open dense $\sigma$-compact set $U$ in $X$ and $C_0(X)$ has the countable sup property, then $X$ is CCC.
\end{prop}

\begin{proof}
Pick a disjoint family $(U_\lambda)$ of open sets of $X$. By Urysohn's lemma for each $\lambda$ there exists a nonzero nonnegative function  $f_\lambda\in C_0(X)$ with $\supp f_\lambda\subseteq U_\lambda.$ Hence, $(f_\lambda)$ is a disjoint system of nonnegative functions in $C_0(X)$. Since $C_0(X)$ has a weak unit by \Cref{C_0 CS}, \Cref{countable sup weak unit} implies that $(f_\lambda)$ is countable. Hence $(U_\lambda)$ is countable and $X$ is CCC.
\end{proof}

In particular, when $X$ is locally compact $\sigma$-compact Hausdorff space, then $X$ is CCC \Iff $C_0(X)$ has the countable sup property.
If $C_0(X)$ has the countable sup property, then $C_c(X)$ as a  sublattice of $C_0(X)$ has the property itself. It goes the other way around when $X\in T_{3 \frac 1 2}$:

\begin{thm}\label{CCC c_0 c_c}
For $X\in T_{3 \frac 1 2}$ the following assertions are equivalent.
\begin{enumerate}
\item[(a)] $C_c(X)$ has the countable sup property.
\item[(b)] $C_0(X)$ has the countable sup property.
\item[(c)] For each compact set $K\subseteq X$ there exists a relatively compact cozero set $U$ such that $K\subseteq U$ and $C(\overline U)$ has the countable sup property.
\item[(d)] $X$ can be covered by relatively compact cozero sets which satisfy CCC.
\end{enumerate}
\end{thm}

\begin{proof}
(a)$\Rightarrow$(b)
Pick a disjoint system $(f_\lambda)$ of functions in $C_0(X)$ such that  $0\leq f_\lambda\leq f$ for some $f\in C_0(X)$ and each $\lambda.$
For every $n\in\NN$ we define $K_n=f^{-1}([\tfrac 1 n,\infty))$, $U_n=f^{-1}((\tfrac 1 n,\infty))$
and $\Lambda_n=\{\lambda\in\Lambda:\; f_\lambda^{-1}((0,\infty))\cap K_n\ne\emptyset\}$. Since $f_n\in C_0(X)$, sets $K_n$ and $U_n$ are compact and open in $X$, respectively.
Let us fix $n\in\NN$.
Since $K_n\subset U_{n+1}$, for every $x\in K_n$ there exists $\varphi_x\colon X\to [0,1]$ with compact support
such that $\varphi_x(x)=1$ and $\varphi_x(X\setminus U_{n+1})=0$. Hence, $\{\varphi_x^{-1}((0,1]):\; x\in K_n\}$
is an open cover of $K_n$. By compactness of $K_n$ there exists a finite set $F\subset K_n$ such that $K_n\subset \bigcup_{x\in F}\varphi_x^{-1}((0,1])$.
The continuous function $\varphi:=\sum_{x\in F}\varphi_x$  satisfies $\varphi(x)>0$ for all $x\in K_n$ and $\varphi(X\setminus U_{n+1})=\{0\}$. Since $U_{n+1}$ is relatively compact, we conclude $\varphi\in C_c(X)$. Then $\{f_\lambda\cdot\varphi:\; \lambda\in\Lambda_n\}$ is a disjoint family of nonnegative functions satisfying  $0\leq f_\lambda\cdot\varphi\leq f\cdot\varphi\in C_c(X)$.  By the assumption  the set  $\Lambda_n$ is countable and since $\Lambda=\bigcup_{n=1}^\infty \Lambda_n$ we finally conclude that $\Lambda$ is countable. That $C_0(X)$ has the countable sup property follows from \Cref{countable sup disjoint system}.

(a)$\Rightarrow$(c)
Pick a compact set $K$ and find a nonnegative function $\varphi\in C_c(X)$ with $\varphi\equiv 1$ on $K$. The set
$U=\{x\in X:\; \varphi(x)>0\}$ is a cozero set with $\overline U=\supp \varphi$ compact.

We claim that $C(\overline U)$ has the countable sup property.
To prove this, pick a disjoint system $(f_\lambda)$ of nonnegative functions in $C(\overline U)$ which is bounded from above by $f\in C(\overline U)$.
For each $\lambda$ we define a function $g_\lambda\colon X\to \mathbb R$ by
$$g_\lambda(x)=\left\{\begin{array}{ccl}
\varphi(x)f_\lambda(x)&:& x\in\overline U\\
0&:& x\in X\setminus U
\end{array}\right..$$
If $x\in \overline U\setminus U$, then $\varphi(x)=0$. Hence, the restrictions of $g_\lambda$ to $\overline U$ and $X\setminus U$, respectively, agree. Since they are continuous and $\{\overline U,X\setminus U\}$ is a closed cover of $X$, we conclude $g_\lambda\in C_c(X)$. By a similar argument we see that  $g:X\to \mathbb R$ defined as
$$g(x)=\left\{\begin{array}{ccl}
\varphi(x)f(x)&:& x\in\overline U\\
0&:& x\in X\setminus U
\end{array}\right.$$
is also in $C_c(X)$.
By the assumption the family $(g_\lambda)$ is countable. From continuity of functions $f_\lambda$ and from  $\varphi>0$ on $U$ we conclude that $(f_\lambda)$ is also countable.

(c)$\Rightarrow$(d) Pick $x\in X$ and find a relatively compact cozero set $U$ such that $x\in U$ and $C(\overline U)$ has the countable sup property.  By  \Cref{CCC = countabe sup} the space $\overline U$ satisfies CCC and since $U$ is dense in $\overline U$, $U$ itself satisfies CCC.

(d)$\Rightarrow$(a) Pick a positive disjoint system $(f_\lambda)$ in $C_c(X)$ such that $0\leq f_\lambda \leq f$ for some $f\in C_c(X)$ and each $\lambda.$ Since the support $\supp f$ of  $f$ is compact,  there exist cozero sets $U_1,\ldots,U_n$ which satisfy CCC and $\supp f\subseteq U_1\cup \cdots \cup U_n.$ Since the set $U:=U_1\cup \cdots \cup U_n$  also satisfies CCC, the family $(f_\lambda|_U)$ of restrictions to $U$ is countable. To finish the proof note that each function $f_\lambda$ is zero on $X\setminus U$.
\end{proof}

Here we would like to point out two things. First, if $X$ is an uncountable discrete space, then $C_0(X)$ has the countable sup property while $C(X)$ does not. Second, when $X\in T_{3\frac 1 2}$, then $C_0(X)$ is the norm completion with respect to the the uniform topology of its order ideal $C_c(X)$. This leads us to the following question.

\begin{quest}
 Suppose $E$ is a locally solid vector lattice with the countable sup property. Does the topological completion $\widehat E$ of $E$ also have the countable sup property?
\end{quest}

The answer is yes when $E$ is metrizable and sits in $\widehat E$ as an order ideal.

\begin{thm}
Let $E$ be a metrizable locally solid vector lattice which is an ideal in its topological completion $\widehat E$. Then $E$ has the countable sup property \Iff $\widehat E$ has the countable sup property.
\end{thm}

\begin{proof}
Assume that $E$ has the countable sup property, pick a nonzero vector $0\leq x\in \widehat E$ and a positive disjoint system $(x_\lambda)_{\lambda\in\Lambda}$ which satisfies $0\leq x_\lambda\leq x$ for each $\lambda.$
 Find a sequence $(y_n)$ in $E$ with $y_n\to x$ in $\widehat E$. Since the lattice operations are continuous, by successively replacing $y_n$ first with $y_n^+$ and then with $y_n\wedge x$ we may assume that $0\leq y_n\leq x$ for each $n\in\mathbb N$. Since $0\leq y_n\wedge x_\lambda\leq y_n\wedge x$ in $E$,
 the disjoint system $(y_n\wedge x_\lambda)_{\lambda\in\Lambda}$ is countable in $E$ for each $n\in\mathbb N$. Hence, the set
 $$\{(n,\lambda)\in \mathbb N\times \Lambda:\; y_n\wedge x_\lambda\neq 0\}$$ is countable. Since $y_n\to x$, for each $\lambda\in\Lambda$ there is at least one $n(\lambda)\in \mathbb N$ such that $y_{n(\lambda)}\wedge x_\lambda\neq 0$, so that $\Lambda$ is countable. By \Cref{countable sup disjoint system} we conclude that $\widehat E$ has the countable sup property.
\end{proof}

\begin{ex}
Let $\Lambda$ be an uncountable set equipped with the discrete topology. We equip the vector lattice $C_c(\Lambda)$ by the topology of pointwise convergence which is not metrizable  since $\Lambda$ is uncountable. The topological completion of $C_c(\Lambda)$ is the vector lattice $\mathbb R^\Lambda$ of all real valued functions on $\Lambda$ which does not have the countable sup property while $C_c(\Lambda)$ itself has the property.
\end{ex}

\section{The countable sup property of the lattice $C(X\times Y)$} \label{ccc product}

The product of two CCC spaces is not necessary a CCC space. We will see that the product of two
vector lattices $E_1$ and $E_2$ with the countable sup property has the same property with respect to both natural
orders on $E_1\times E_2$. More interesting question is if a vector lattice $C(X\times Y)$ has the
countable sup property, provided the lattices $C(X)$ and $C(Y)$ have the property.

Let $E_1$ and $E_2$ be vector lattices. On the product $E_1\times E_2$ of the sets $E_1$ and $E_2$ we can
naturally define two orders:
\begin{itemize}
\item $(x_1,y_1)\leq (x_2,y_2)$ iff $x_1\le x_2$ and $y_1\le y_2$,
\item $(x_1,y_1)\preceq (x_2,y_2)$ iff $x_1<x_2$ or ($x_1=x_2$ and $y_1\le y_2$).
\end{itemize}
The second one is \term{lexicographical order}.

\begin{prop}
If $E_1$ and $E_2$ have the countable sup property then both lattices
$(E_1\times E_2,\leq)$ and $(E_1\times E_2,\preceq)$ have the countable sup property.
\end{prop}

\begin{proof}
Let $0\le (x_\alpha,y_\alpha)\uparrow (x,y)$ in $(E_1\times E_2,\leq)$.
Then $0\le x_\alpha\uparrow x$ in $E_1$ and $0\le y_\alpha\uparrow y$ in $E_2$.
Since $E_1$ and $E_2$ satisfy the countable sup property there exist
increasing sequences $(\alpha_n)$ and $(\beta_n)$ such that
$0\le x_{\alpha_n}\uparrow x$ in $E_1$ and $0\le y_{\beta_n}\uparrow y$ in $E_2$.
Choose arbitrary $\gamma_1>\alpha_1,\beta_1.$ Inductively, for each $n\in\mathbb N$ one can find $\gamma_{n+1}>\alpha_{n+1},\beta_{n+1},\gamma_n.$
Then $\gamma_n$ is an increasing sequence, $0\le x_{\gamma_n}\uparrow x$ in $E_1$, and $0\le y_{\gamma_n}\uparrow y$ in $E_2$.
 This yields $0\le (x_{\gamma_n},y_{\gamma_n})\uparrow (x,y)$ in $(E_1\times E_2,\leq)$.

Let $(x_\alpha,y_\alpha)_{\alpha\in\Lambda}$ be a net such that
$0\le (x_\alpha,y_\alpha)_{\alpha\in\Lambda}\uparrow (x,y)$ in $(E_1\times E_2,\preceq)$.
If there is no $\alpha\in\Lambda$ such that $x_\alpha=x$, then $(x_\alpha,y_\alpha)<(x,z)$ for all $\alpha$ and all $z\in E_2$.
Therefore $y$ is the first element in $E_2$  which is a contradiction. Hence, the set $\widetilde\Lambda=\{\alpha\in\Lambda :\; x_\alpha=x\}$ is nonempty.
From $0\le (x,y_\alpha)_{\alpha\in\widetilde\Lambda}\uparrow (x,y)$
we conclude $0\le y_\alpha \uparrow_{\alpha\in\widetilde\Lambda}y$.
Since $E_2$ has the countable sup property, there exists an increasing sequence
$(\alpha_n)$ in $\widetilde{\Lambda}$ such that $0\le y_{\alpha_n}\uparrow y$ in $E_2$. Finally we obtain
$0\le (x_{\alpha_n},y_{\alpha_n})\uparrow (x,y)$ in $(E_1\times E_2,\preceq)$.
\end{proof}

So the product of function spaces $C(X)$ and $C(Y)$ with the countable sup property has the same property
in both mentioned cases. The situation with the product of the form $C(X\times Y)$ is more difficult. We do not know whether $C(X\times Y)$ has the countable sup property whenever $C(X)$ and $C(Y)$ have it. It seems one needs to assume some special axioms to answer the question. This is not so surprising since the countable sup property of $C(X)$ is tightly connected to CCC of the space $X$.
If we assume the Continuum Hypothesis, there exists a compact Hausdorff space $X$ which is CCC but $X\times X$ is not (see e.g. \cite{Roy:89}). Since compact Hausdorff spaces are $T_{3\frac 1 2}$, this space $X$  by \Cref{CCC = countabe sup} serves us also as an example of a space such that $C(X)$ has the countable sup property while $C(X\times X)$ does not have it.

If we sharpen an assumption on one of the factor $X$ and $Y$, we have a positive answer.

\begin{thm}\label{csp + cal -> csp}
If $C(X)$ has the countable sup property and $Y$ has calibre $\aleph_1$, then $C(X\times Y)$ has the countable sup property.
\end{thm}

\begin{proof}
Suppose there exists an uncountable family $\{f_\lambda\colon X\times Y\to [0,1]:\; \lambda\in\Lambda\}$ of pairwise disjoint functions in $C(X\times Y)$.
For every $\lambda$ there exists an open set $U_\lambda\times V_\lambda\subset f_\lambda^{-1}((0,1])$.
Because $Y$ has calibre $\aleph_1$ there exist an uncountable subset $\Lambda'\subset\Lambda$ and $y_0\in Y$
such that $y_0\in\bigcap_{\lambda\in\Lambda'} V_\lambda$. Let us define $g_\lambda(x)=f_\lambda(x,y_0)$.
Then $\{g_\lambda :\; \lambda\in\Lambda'\}$ is an uncountable family of nonzero pairwise disjoint functions in $C(X)$
which is a contradiction. Hence $C(X\times Y)$ has the countable sup property.
\end{proof}

Since every separable space has calibre $\aleph_1$, the following corollary can be seen as a vector lattice version of \cite[Theorem 3.3]{Wiscamb:69}.

\begin{cor}\label{ccc + sep -> csp}
If $C(X)$ has the countable sup property and $Y$ is separable, then $C(X\times Y)$ has the countable sup property.
\end{cor}

It is known that the product of a CCC space and a space satisfying Knaster's condition is a CCC space. In particular, if $X$ is CCC and $Y$ satisfies Knaster's condition, then $C(X\times Y)$ has the countable sup property.
However, we do not know if $C(X\times Y)$ has the countable sup property provided $C(X)$ has the countable sup property and $Y$ satisfies Knaster's condition.

The following theorem is a vector lattice version of \cite[Theorem 2.2]{Roy:89}.

\begin{thm}\label{koncni produkti cs}
Let $(X_\lambda)_{\lambda\in\Lambda}$ be a family of topological spaces such that  for each finite family $\Lambda_0\subseteq \Lambda$ the space  $C(\prod_{\lambda\in\Lambda_0} X_\lambda)$ has the countable sup property, then
$C(\prod_{\lambda\in\Lambda} X_\lambda)$ has the countable sup property.
\end{thm}

\begin{proof}
Suppose there exists an uncountable family of pairwise disjoint functions
$\{f_\alpha\colon \prod_{\lambda\in\Lambda} X_\lambda\to [0,1]:\; \alpha\in A\}$
in $C(\prod_{\lambda\in\Lambda} X_\lambda)$.
For every $\alpha\in A$ there is a basis set $U_\alpha\subset f_\alpha^{-1}( (0,1])$, i.e.,
there exists a finite set $\Lambda_\alpha\subset \Lambda$ and for every $\lambda\in \Lambda_\alpha$
there exists an open set $V_\lambda^\alpha\subset X_\lambda$ such that
$U_\alpha=\bigcap_{\lambda\in \Lambda_\alpha} p_\lambda^{-1}(V_\lambda^\alpha)$,
where $p_\lambda\colon \prod_{\mu\in\Lambda} X_\mu\to X_\lambda$ is the projection.
By delta system lemma \cite[p.~174]{Roy:89} there exists an uncountable subset $\widetilde A\subseteq A$
and a set $\Lambda_0\subset\Lambda$ such that
$\Lambda_\alpha\cap \Lambda_\beta=\Lambda_0$ for all $\alpha,\beta\in\widetilde A$
for which $\Lambda_\alpha\ne \Lambda_\beta$. If $\Lambda_\alpha\cap \Lambda_\beta=\emptyset$
for $\alpha,\beta\in\widetilde A$ the functions
$f_\alpha$ and $f_\beta$ are not disjoint, hence $\Lambda_0\ne\emptyset$.
For every $\lambda\not\in\Lambda_0$ there is at most one $\alpha\in\widetilde A$ such that $\lambda\in \Lambda_\alpha$.
Pick an element $x_\lambda\in V_\lambda^\alpha$ if there is $\alpha\in\widetilde A$ such that $\lambda\in \Lambda_\alpha$
otherwise pick any $x_\lambda\in X_\lambda$. Let
$i\colon \prod_{\lambda\in \Lambda_0} X_\lambda\to \prod_{\lambda\in\Lambda} X_\lambda$ be the slice embedding
such that $p_\lambda i(x)=x_\lambda$ for every $\lambda\not\in\Lambda_0$.
Then $f_\alpha\circ i$ is nonzero for every $\alpha\in\widetilde A$ and
$\{f_\alpha\circ i:\; \alpha\in\widetilde A\}$ is an uncountable family of pairwise disjoint functions
in $C(\prod_{\lambda\in \Lambda_0} X_\lambda)$ which is a contradiction.
\end{proof}

\begin{cor}
Let $(X_\lambda)_{\lambda\in\Lambda}$ be a family such that $C(X_{\lambda_0})$ has the countable sup property for some $\lambda_0\in\Lambda$
and $X_\lambda$ has calibre $\aleph_1$ for all $\lambda\in\Lambda\setminus\{\lambda_0\}$. Then $C({\prod_{\lambda\in\Lambda} X_\lambda})$ has the countable sup property.
\end{cor}

\begin{proof}
Pick a finite subset $\Lambda_0\subseteq \Lambda$.
If $\lambda_0\not\in \Lambda_0$, then $\prod\limits_{\lambda\in\Lambda_0}X_\lambda$ has calibre $\aleph_1$, hence $C(\prod\limits_{\lambda\in\Lambda_0}X_\lambda)$ has the countable sup property.
If $\lambda_0\in \Lambda_0$, by induction and \Cref{csp + cal -> csp} the space $C(\prod\limits_{\lambda\in\Lambda_0}X_\lambda)$ has the countable sup property.
To finish the proof we apply \Cref{koncni produkti cs}.
\end{proof}

Of course we can replace the property ``calibre $\aleph_1$" with stronger property ``separability" in the above Corollary
to get the same result.


\begin{thebibliography}{WWWW}

\bibitem[AB03]{Aliprantis:03}
C.D.~Aliprantis and O.~Burkinshaw, \emph{Locally solid {R}iesz
  spaces with applications to economics}, 2nd ed., AMS, Providence, RI, 2003.

  \bibitem[AB06]{Aliprantis:06}
C.~D. Aliprantis and O.~Burkinshaw, \emph{Positive operators}, Springer,
  Dordrecht, 2006, Reprint of the 1985 original.

\bibitem[AT17]{Adeeb:17}
S.~Adeeb, V.~G.~Troitsky, \emph{Locally piecewise affine functions and their order structure}, Positivity \textbf{21} (2017), 213--221.


\bibitem[Dug66]{Dugundji:66}
J.~Dugundji, \emph{Topology}, Allyn and Bacon, Inc., Boston, Mass. 1966.

\bibitem[Gal80]{Galvin}
F.~Galvin, \emph{Chain conditions and products}, Fund. Math. \textbf{108} (1980), 33--48.

\bibitem[Gao14]{Gao:14}
N.~Gao, Unbounded order convergence in dual spaces,
\emph{J.\ Math.\ Anal.\ Appl.}, 419, 2014, 347--354.

\bibitem[GLX]{GLX}
N.~Gao, D.H.~Leung, and F.~Xanthos,
Dual representation problem of risk measures.
arXiv:1610.08806 [q-fin.MF]

\bibitem[GLX2]{GLX2}
N.~Gao, D.H.~Leung, and F.~Xanthos,
\emph{Duality for unbounded order convergence and applications},
arXiv:1705.06143 [math.FA]

\bibitem[GTX]{GTX}
N.~Gao, V.G.~Troitsky, and F.~Xanthos,
\emph{Uo-convergence and its applications to Ces\`aro means in Banach lattices}, Isr. J. Math. (2017). doi:10.1007/s11856-017-1530-y

\bibitem[GX14]{GaoX:14}
N.~Gao and F.~Xanthos,
Unbounded order convergence and application to martingales without
probability,
\emph{J.\ Math.\ Anal.\ Appl.}, 415 (2014), 931--947.
\bibitem[Jech67]{Jech}
T.~Jech, \emph{Non-provability of Souslin's hypothesis}, Comment. Math. Univ. Carolinae \textbf{8} (1967), 291--305.

\bibitem[KMT17]{KMT}
M. Kandi\'c, M. Marabeh, V.G. Troitsky, Unbounded norm topology in
Banach lattices, \emph{J.\ Math.\ Anal.\ Appl.}, 451 (2017), no. 1,
259–-279.

\bibitem[KV]{KV}
M.~Kandi\'c, A.~Vavpeti\v c, \emph{Topological aspects of order in $C(X)$}, arXiv:1612.05410 [math.FA]


\bibitem[LC]{LiChen}
H.~Li, Z.~Chen, \emph{Some loose ends on unbounded order convergence}, Positivity (2017). doi:10.1007/s11117-017-0501-1

\bibitem[LZ71]{Luxemburg:71}
W.A.J.~ Luxemburg, A.C.~Zaanen, \emph{Riesz spaces I},
North-Holland Mathematical Library, 1971.

\bibitem[MS70]{Martin}
D.~A.~Martin, R.~M.~Solovay, \emph{Internal Cohen extensions},
Ann. Math. Logic \textbf{2} (1970), 143--178.


\bibitem[MN91]{Meyer-Nieberg:91}
P.~Meyer-Nieberg,
\emph{Banach lattices},
Springer-Verlag, Berlin, 1991.

\bibitem[Roy89]{Roy:89}
N.~M.~Roy, \emph{Is the product of CCC spaces a CCC space?},  Publ. Mat. \textbf{33} (1989), no. 2, 173--183

\bibitem[ST71]{Solovay}
R.~M.~Solovay, S.~Tennenbaum, \emph{Iterated Cohen extensions and Souslin's problem}, Ann. of Math. (2) \textbf{94} (1971), 201--245.

\bibitem[SS70]{Steen:70}
L.~A.~Steen, J.~A.~Seebach, \emph{Counterexamples in Topology}, Holt, Rinehart and Winston, New York-Montreal, Que.-London, 1970.

\bibitem[Sus20]{Suslin}
M. Souslin, \emph{Probl\`{e}me 3}, Fund. Math. 1 (1920), 223.

\bibitem[Ten68]{Tennenbaum}
S.~Tennenbaum, \emph{Souslin's problem}, Proc. Nat. Acad. Sci U.S.A. \textbf{59} (1968), 60--63.

\bibitem[Tka98]{Tkachenko}
M.~Tka\v cenko, \emph{Introduction to topological groups}, Topology ad its Applications \textbf{86} (1998), 179--231.

\bibitem[Tod97]{Tod:97}
S.~Todorcevic, \emph{Topics in Topology}, Lecture Notes in Math. 1652 (Springer, Berlin, 1997).

\bibitem[VR59]{VR:59}
V.S.~Varadarajan, R.~Ranga~Rao, \emph{On a theorem in metric spaces}, Ann. Math. Stat., \textbf{29} (1959), 612--613.

\bibitem[Wis69]{Wiscamb:69}
M.~R.~Wiscamb, \emph{The discrete countable chain condition}, Proc. Amer. Math. Soc. \textbf{23} 1969 608--612.


\end{thebibliography}
\end{document}